\documentclass{amsart}

\usepackage{etex}
\usepackage{amsmath}
\usepackage{amssymb}
\usepackage{amscd}
\usepackage{amsthm}
\usepackage{enumerate}
\usepackage{array}
\usepackage{makecell}
\usepackage{epsf} 
\usepackage{pictexwd}
\usepackage{rotating}
\usepackage{pdflscape}
\usepackage{tikz}
\usepackage{float}  \restylefloat{figure} 
\usepackage[OT2,OT1]{fontenc}

\theoremstyle{plain}
\newtheorem{thm}{Theorem}

\newtheorem{prop}[thm]{Proposition}

\theoremstyle{definition}
\newtheorem*{mydef}{Definition}

\theoremstyle{remark}

\let\myskip=\medskip

\def\st{\,\,\big|\,\,}
\def\bs{\backslash}
\def\<{\langle}
\def\>{\rangle}

\let\ge=\geqslant
\let\le=\leqslant
\let\emptyset=\varnothing

\def\definebb#1=#2.{\def#1{{{\mathbb #2}^{\vphantom{x}}}}}
\definebb \c=C.
\definebb \r=R.
\definebb \d=D.
\definebb \n=N.
\definebb \z=Z.
\definebb \q=Q.

\def\calF{{\mathcal F}}
\def\calE{{\mathcal E}}
\DeclareMathOperator{\Cl}{Cl}
\DeclareMathOperator{\Int}{Int}
\DeclareMathOperator{\Isom}{Isom}
\DeclareMathOperator{\lcm}{lcm}

\DeclareMathOperator{\PSU}{PSU}

\DeclareMathOperator{\SL}{SL}
\DeclareMathOperator{\SU}{SU}
\let\Im=\undefined \DeclareMathOperator{\Im}{Im}
\let\Re=\undefined \DeclareMathOperator{\Re}{Re}

\def\dd{\partial}

\def\al{{\alpha}}

\def\la{{\lambda}}
\def\Ga{{\Gamma}}

\def\thk{{\frac{\pi k}{2p}}}
\def\ellE{{\ell}}
\def\bGa{\overline\Ga}

\def\Grp{{\bar G}}
\def\tGrp{G}
\def\Lrp{{\bar L}}
\def\tLrp{L}

\def\WZZW{\begin{pmatrix} \bar w&z \\ \bar z&w \end{pmatrix}}

\def\Eg{E_g}
\def\Ig{I_g}
\def\Hg{H_g}
\def\Ee{E_e}
\def\Fe{F_e}

\def\FE{F_{\calE}}

\def\bE{{\bar E}}
\def\bI{{\bar I}}

\def\Qx{Q_x}
\def\Qu{Q_u}

\def\Gau{{\Ga_1(u)}}
\def\Gauou{{\Gau\backslash\{u\}}}

\def\cupQxGauou{{\bigcup\limits_{\hbox to 0pt{\hss$\scriptstyle x\in\Gauou$\hss}}\Qx}}

\begin{document}

\author[Nasser Bin Turki]{Nasser Bin Turki}
\address{Department of Mathematics, College of Science, King Saud University, P.O.~Box 2455, Riyadh 11451, Saudi Arabia}
\email{nassert@ksu.edu.sa}

\author[Anna Pratoussevitch]{Anna Pratoussevitch}
\address{Department of Mathematical Sciences\\ University of Liverpool\\ Peach Street \\ Liverpool L69~7ZL}
\email{annap@liv.ac.uk}

\title[Two series of Lorentz polyhedral fundamental domains]{Two series of polyhedral fundamental domains for Lorentz bi-quotients}

\begin{date} {\today} \end{date}

\thanks{The authors extend their appreciation to the Deanship of Scientific Research at King Saud University for funding this work through the research group No (RG-1440-142).
To produce the images of polyhedra we used Geomview, a programme developed at the Geometry Centre, University of Minnesota, Minneapolis.}


\begin{abstract}
The main aim of this paper is to give two infinite series of examples of Lorentz space forms
that can be obtained from Lorentz polyhedra by identification of faces.
These Lorentz space forms are bi-quotients of the form  $\Ga_1\backslash\tGrp/\Ga_2$,
where $\tGrp=\widetilde{\operatorname{SU}(1,1)}\cong\widetilde{\operatorname{SL}(2,\r)}$
is a simply connected Lie group with the Lorentz metric given by the Killing form,
$\Ga_1$ and $\Ga_2$ are discrete subgroups of~$\tGrp$ and $\Ga_2$ is cyclic.
A construction of polyhedral fundamental domains for the action of $\Ga_1\times\Ga_2$ on~$\tGrp$ via $(g,h)\cdot x=gxh^{-1}$ was given in the earlier work of the second author.
In this paper we give an explicit description of the fundamental domains obtained by this construction for two infinite series of groups.
These results are connected to singularity theory as the bi-quotients $\Ga_1\backslash\tGrp/\Ga_2$ appear as links of certain quasi-homogeneous $\q$-Gorenstein surface singularities,
i.e.\ the intersections of the singular variety with sufficiently small spheres around the isolated singular point.
\end{abstract}

\subjclass[2000]{Primary 53C50; Secondary 14J17, 32S25, 51M20, 52B10}






\keywords{Lorentz space form, polyhedral fundamental domain, quasi-homogeneous singularity.}

\maketitle

\section{Introduction}

\label{intro}

\noindent
In this paper we study the Lorentz space forms obtained as bi-quotients $\Ga_1\bs\tGrp/\Ga_2$,
where $\tGrp$ is the universal cover
$$\tGrp=\widetilde{\SU(1,1)}\cong\widetilde{\SL(2,\r)},$$
$\Ga_1$ and $\Ga_2$ are discrete subgroups of~$\tGrp$ of the same finite level, and $\Ga_2$ is cyclic.
The group $\Ga_1\times\Ga_2$ acts on $\tGrp$ via~$(g,h)\cdot x=gxh^{-1}$.
The {\sl level} of a subgroup $\Ga$ of $\tGrp$ is the index of $\Ga\cap Z(\tGrp)$ in the centre $Z(\tGrp)\cong\z$.
Subgroups of~$\tGrp$ of finite level~$k$ are those subgroups that can be obtained as pre-images of subgroups of the $k$-fold covering of~$\PSU(1,1)$.
We use the construction introduced in~\cite{Pr:qgor-fund} of polyhedral fundamental domains for this action 
which generalizes the results of~\cite{Pr:statia}, \cite{BPR}, \cite{BKNRS}.
This construction leads to a description of the bi-quotients $\Ga_1\bs\tGrp/\Ga_2$
as polyhedra with faces identified according to some gluing rules on the boundary.

\myskip
Let $\Ga(p_1,3,3)^k$ be a subgroup of $\tGrp$ of level~$k$ such that its image under the projection to~$\PSU(1,1)$ is a triangle group $\Ga(p_1,3,3)$,
see section~\ref{triangles} for more details.
Subgroups of~$\tGrp$ and~$\PSU(1,1)$ act on the disc model~$\d$ of the hyperbolic plane since $\Isom^+(\d)\cong\PSU(1,1)$.
Let $u\in\d$ be the fixed point of a generator of~$\Ga(p_1,3,3)$ of order~$p_1$.
Let $(C_{p_2})^k$ be a subgroup of $\tGrp$ of level~$k$ such that its image under the projection to~$\PSU(1,1)$ is a cyclic group of order~$p_2$ generated by an elliptic element with fixed point~$u$.
The aim of this paper is to give explicit descriptions of fundamental domains for two infinite series of groups $\Ga(k+3,3,3)^k\times(C_3)^k$ and $\Ga(2k+3,3,3)^k\times(C_3)^k$,
where $k$ is a positive integer not divisible by~$3$.

\myskip
The bi-quotients $\Ga_1\bs\tGrp/\Ga_2$ are the links of certain quasi-homogeneous $\q$-Go\-ren\-stein surface singularities as explained in~\cite{Pr:qgor}, see also~\cite{Dol83}.
The motivation for the choice of the series $\Ga(k+3,3,3)^k\times(C_3)^k$ and $\Ga(2k+3,3,3)^k\times(C_3)^k$ is that they correspond
to $\q$-Gorenstein surface singularities that are obtained as quotients of certain singularities in the series~$E$ and~$Z$ according to the classification by V.I.~Arnold.
We listed in the table below the normal form of the singularity as well as the level~$k$ of~$\Ga_i$ and the signature~$(p,q,r)$ of the image of~$\Ga_1$ in $\PSU(1,1)$, 
for more details see \cite{AGZV}, \cite{Dol74}, \cite{Dol75} and \cite{Pr:diss}.

\begin{table}[h]
\centering
\begin{center}
    \begin{tabular}{|c|c|c|c|}
     \hline
     level & $(p,q,r)$ & type & normal form \\
     \hline
     &&&\\
     $k$ & $(k+3,3,3)$ &$E_{4k+10}$ &  $x^3+y^{2k+6}+ z^2$ \\
     &&&\\
     \hline
     &&&\\
     $k$ & $(2k+3,3,3)$ &$Z_{4k+9}$ &  $x^3y+y^{2k+4}+ z^2$ \\
     &&&\\
     \hline
    \end{tabular}
\end{center}
\label{tab:Singularity}
\end{table}

\myskip\noindent
We will now outline the fundamental domain construction described in~\cite{Pr:qgor-fund}.
Suppose that $\Ga_1$ and~$\Ga_2$ are discrete subgroups of level~$k$ in~$\tGrp$ and $\Ga_2$ is cyclic.
We assume the existence of a joint fixed point $u\in\d$ of~$\Ga_1$ and~$\Ga_2$.
Let $p_i=|(\bGa_i)_u|$ be the size of the isotropy group of~$u$ in the image~$\bGa_i$ of $\Ga_i$ in~$\PSU(1,1)$ for~$i=1,2$.
This construction works if~$p=\lcm(p_1,p_2)>k$.

\myskip
We think of the group $\tGrp$ as a hypersurface embedded in the bundle $\tLrp=\r_+\times\tGrp$.
The Killing form on $\tGrp$ induces a pseudo-Riemannian metric of signature $(2,1)$ on~$\tGrp$ and of signature~$(2,2)$ on $\tLrp$.
The hypersurface~$\tGrp$ is replaced with its piece-wise totally geodesic polyhedral model, specially adapted to the action of $\Ga_1\times\Ga_2$.
The fundamental domains~$F_{g_1g_2}$, $g_1\in\Ga_1$, $g_2\in\Ga_2$ for the action of $\Ga_1\times\Ga_2$ on the model polyhedron are its faces,
which are then projected onto~$\tGrp$ to obtain fundamental domains~$\calF_{g_1g_2}$ for the action of $\Ga_1\times\Ga_2$ on~$\tGrp$.
Similar ideas of projecting an affine construction with half-planes onto a quadric were used in the algorithmic construction of Voronoi diagrams for point sets
in the Euclidean and hyperbolic plane, compare~\cite{BY}.

\myskip
For $g\in\tGrp$ let $\Eg$ be the embedded tangent space of~$\tGrp$ at the point~$g$,
i.e.\ the $3$-dimensional totally geodesic submanifold of~$\tLrp$ which is tangent to~$\tGrp$ at the point~$g$.
The hypersurface~$E_g$ divides~$\tLrp$ into two connected components.
We will refer to the closures of these connected components as half-spaces and denote them by $\Hg$ and $\Ig$,
see section~\ref{elements} for more details.

\myskip
Let $e$ be the identity in~$\tGrp$.
The interior of the fundamental domain~$\Fe$ can be described as
$$\Fe^{\circ}=(\Ee\cap\dd\Qu)^\circ-\bigcup_{x\in\Gauou}\Qx,$$
where $\Gau\subset\d$ is the orbit of the point~$u$ under the action of~$\Ga_1$ and $\Qx$ for every $x\in\Gau$ is the prism-like polyhedron
$$\Qx=\bigcap\limits_{{(g_1,g_2)\in\Ga_1\times\Ga_2\atop g_1(u)=x}} H_{g_1\cdot g_2}.$$
According to Theorem~B in~\cite{Pr:qgor-fund},
if $\Ga_1$ is co-compact we can replace the union of~$\Qx$ over all $x\in\Gauou$ by the union over some finite subset $N\subset\Gauou$,
which is a crucial step on the way to obtain an explicit description of the fundamental domain~$\Fe$.
However, the proof of this fact uses a compactness argument 
which in general does not provide an explicit description of such a finite subset.
The first main result of this paper is to show that for a certain family of groups such a finite subset of~$\Gauou$ is given by the edge corona of~$\Ga_1$ (see section~\ref{triangles} for the definition).

\begin{thm}
\label{thm-corona}
Let $\Ga_1=\Ga(p_1,3,3)^k$ with $p_1\ge k+3$ and $\Ga_2=(C_3)^k$.
Assume that a generator of~$\Ga_1$ of order~$p_1$ and the cyclic group~$\Ga_2$ have a shared fixed point~$u\in\d$.
Then
$$\Fe^{\circ}=(\Ee\cap\dd\Qu)^\circ-\bigcup_{x\in{\calE}}\Qx,$$
where the finite set~$\calE$ is the edge corona of $\Ga_1(p_1,3,3)$ with respect to~$u$.
\end{thm}


\myskip\noindent
Theorem~\ref{thm-corona} reduces the description of the fundamental domain to an intersection of finitely many unions of finitely many half-spaces.
For two infinite series of groups,
$$\Ga(k+3,3,3)^k\times(C_3)^k\quad\text{and}\quad\Ga(2k+3,3,3)^k\times(C_3)^k,$$
we can reduce the description further after some long but elementary computations with systems of linear inequalities,
for details see~\cite{BT}.
The results are summarized in Theorem~\ref{fund}.
The conjectural description of these fundamental domains for~$k=2$ was given in~\cite{Pr:qgor-fund}.
In this paper, following
~\cite{BT}, we confirm these conjectures and extend them to two infinite series.
We also describe some methods
that can be used to determine the fundamental domains for other series of groups.

\myskip
To state the results we will first introduce some notation.
We make use of the following construction to describe certain elements of~$\tGrp$:
Given a base-point $x\in\d$ and a real number $t$, let $\rho_x(t)\in\PSU(1,1)$ denote the rotation through the angle~$t$ about the point~$x$.
Thus we obtain a homomorphism $\rho_x:\r\to\PSU(1,1)$, which lifts to the unique homomorphism $R_x:\r\to\tGrp$ into the universal covering group.
The element $C=R_u(2\pi)$ is one of the two generators of the centre of~$\tGrp$. 
The stabilisers $(\Ga_1)_u$ and $(\Ga_2)_u=\Ga_2$ generate a cyclic group with generator $D=R_u(2\pi k/p)$, where $p=\lcm(p_1,p_2)$.

\begin{thm}
\label{fund}
Consider the group $\Ga_1\times\Ga_2=\Ga(p,3,3)^k\times(C_3)^k$.
Consider a triangle generating $\Ga(p,3,3)$.
Let $u$ and $v$ be vertices of this triangle with angles~$\pi/p$ and~$\pi/3$ respectively.
Let $\lambda =1 $ if $k=1\mod 3$ and $\lambda =2$ if $k=2\mod 3$.

\begin{enumerate}[$\bullet$]
\item
If $p=k+3$ then
$$F_e=E_e\cap H_D\cap H_{D^{-1}}\cap\bigcap _{m\in\z}(I_{a_m}\cup I_{b_m})$$
is a fundamental domain for $\Ga_1\times\Ga_2$, where
\begin{align*}
  a_0&=R_v(8\pi/3)\cdot D^{2 \lambda p-1}\cdot C^{\frac{-2(\lambda k+2)}{3}},\\
  a_m&=R_u^m(\pi/p)\cdot a_0\cdot R_u^{-m}(\pi/p)
  \quad\text{and}\quad
  b_m=a_m\cdot D.
\end{align*}
\item
If $p=2k+3$ then
$$F_e=E_e\cap H_D\cap H_{D^{-1}}\cap\bigcap_{m\in\z}(I_{a_m}\cup I_{b_m}\cup I_{c_m})$$
is a fundamental domain for $\Ga_1\times\Ga_2$, where
\begin{align*}
  a_0&=R_v(8\pi/3)\cdot D^{2 \lambda p-2}\cdot C^{\frac{-2(\lambda k+2)}{3}},\\
  a_m&=R_u^m(\pi/p)\cdot a_0\cdot R_u^{-m}(\pi/p),\quad
  b_m=a_m\cdot D
  \quad\text{and}\quad
  c_m=b_m\cdot D.
\end{align*}
\end{enumerate}
\end{thm}

\myskip\noindent
The general results in Theorem~\ref{fund} are illustrated by the images of fundamental domains for $\Gamma(k+3,3,3)^k \times (C_3)^k$ and $\Gamma(2k+3,3,3)^k \times (C_3)^k$ for~$k=2$
in Figures~\ref{fund-dom-533-3-2} and~\ref{fund-dom-733-3-2} respectively.
Some explanations of the figures are required.
We project the fundamental domain $F_e$ to the Lie algebra of~$\SU(1,1)$, which is a $3$-dimensional flat Lorentz space of signature $(2,1)$.
The result is a compact polyhedron with flat faces and a distinguished rotational axis of symmetry.
The direction of this axis is negative definite, and the orthogonal complement is positive definite.
Changing the sign of the pseudo-metric in the direction of the rotational axis transforms Lorentz space into a well-defined Euclidean space.
The image of the fundamental domain is then transformed into a polyhedron in Euclidean space with dihedral symmetry.
The direction of the rotational axis is vertical.
The top and bottom faces are removed to make the structure of the polyhedra clearer.


\begin{figure}[h]
\centering
\includegraphics[width=6.35cm]{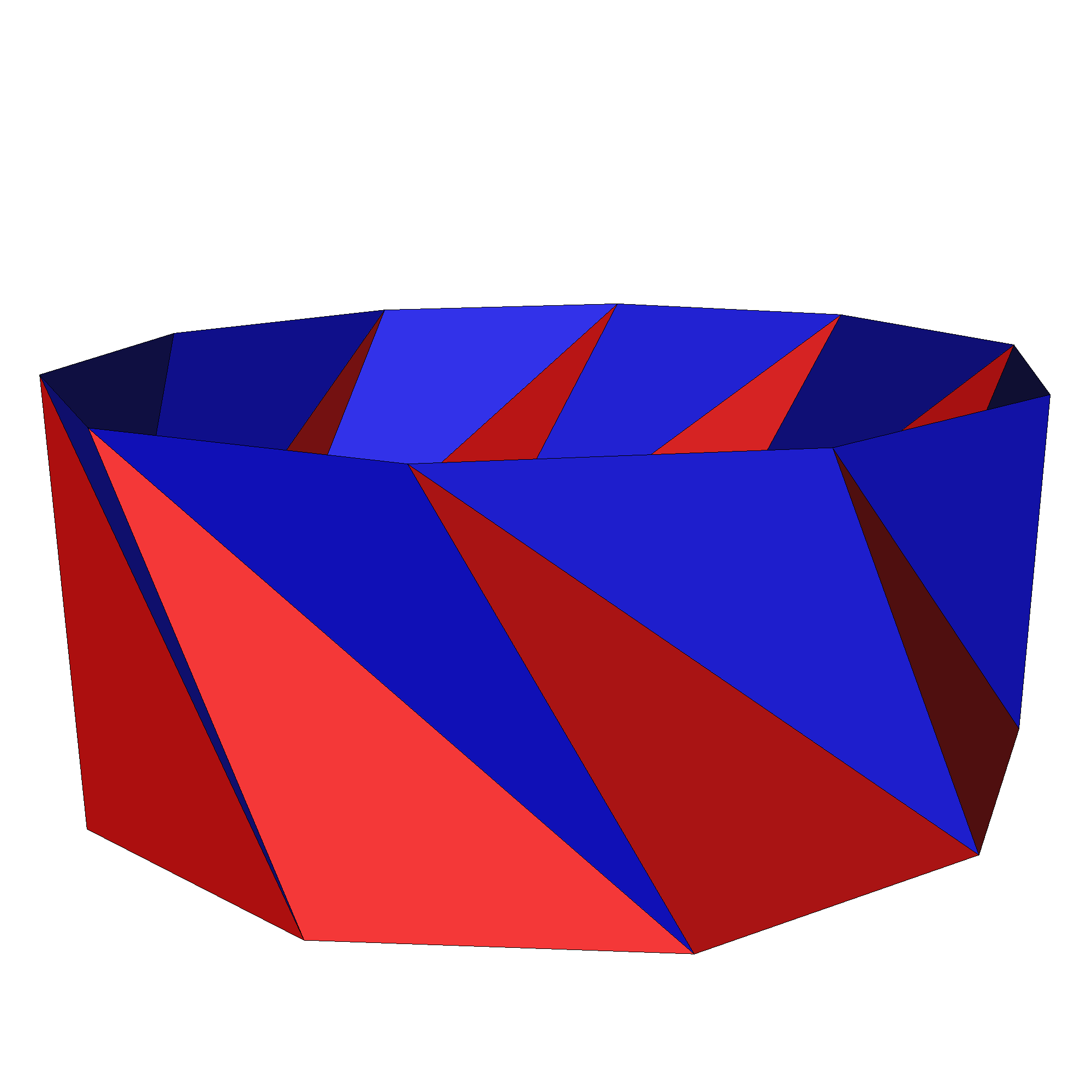}
\caption{Fundamental domain for $\Gamma(5,3,3)^2 \times (C_3)^2$.}
\label{fund-dom-533-3-2}
\end{figure}

\begin{figure}[h]
\centering
\includegraphics[width=6.5cm]{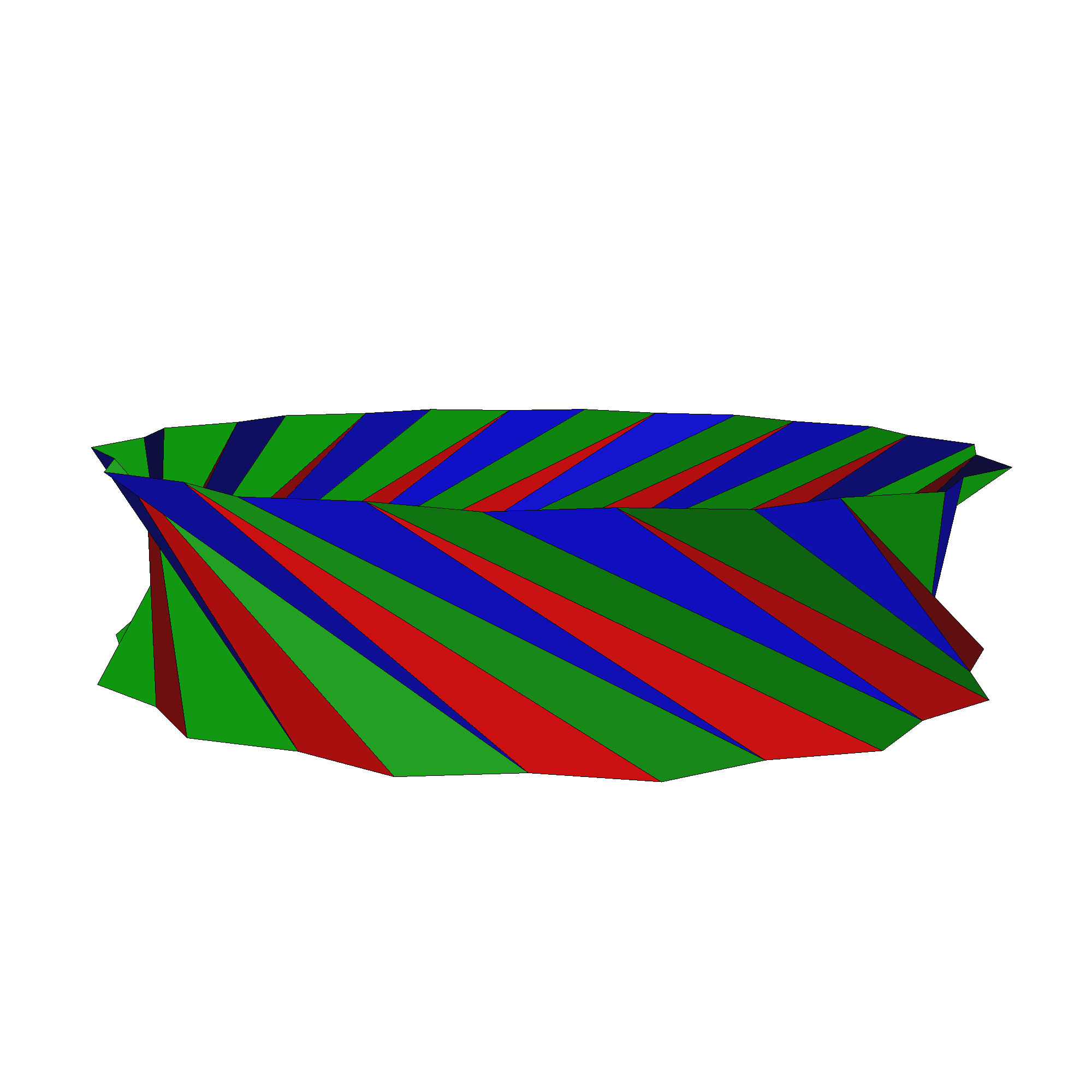}
\caption{Fundamental domain for $\Gamma(7,3,3)^2 \times (C_3)^2$.}
\label{fund-dom-733-3-2}
\end{figure}

\myskip\noindent
Figures~\ref{fig:pkcf} and~\ref{fig:p2kc} illustrate the identification schemes on the boundary of the polyhedron
for the infinite series $\Gamma(k+3,3,3)^k\times(C_3)^k$ and $\Gamma(2k+3,3,3)^k\times(C_3)^k$ respectively.
The face identifications are equivariant with respect to the dihedral symmetry of the polyhedron.
The faces of the same shading/colour are identified.
Numbers on the edges of shaded/coloured faces indicate the identified flags (face, edge, vertex).

\begin{figure}[h]
\centering
\begin{tikzpicture}[scale=0.5]
\draw[color=black,style=thick] (0,0) -- (4,0)node[pos=.5,sloped,below] {$a_{-1}$} -- (0,5) -- (0,0);
\draw[color=black,style=thick, fill=red] (0,5) -- (4,5)  node[pos=.5,below] {3} node[pos=.5,sloped,above] {$b_{-1}$} -- (4,0) node[pos=.5,left] {2} -- (0,5)node[pos=.5,right] {1};
\draw[color=black,style=thick, fill=red] (4,0) -- (8,0) node[pos=.5,above] {1} node[pos=.5,sloped,below] {$a_{0}$} -- (4,5) node[pos=.5,left] {3} -- (4,0)node[pos=.5,right] {2};
\draw[color=black,style=thick, fill=blue] (4,5) -- (8,5) node[pos=.5,below] {3} node[pos=.5,sloped,above] {$b_{0}$} -- (8,0)node[pos=.5,left] {2} -- (4,5) node[pos=.5,right] {1};
\draw[color=black,style=thick, fill=blue] (8,0) -- (12,0) node[pos=.5,above] {1} node[pos=.5,sloped,below] {$a_{1}$} -- (8,5) node[pos=.5,left] {3} -- (8,0)node[pos=.5,right] {2};
\draw[color=black,style=thick ] (8,5) -- (12,5)node[pos=.5,sloped,above] {$b_{1}$} -- (12,0) -- (8,5);
\end{tikzpicture}
\caption{Identifications on the boundary of $\Fe$ in the case $\Ga(k+3,3,3)^k\times(C_3)^k$.}
\label{fig:pkcf}
\end{figure}
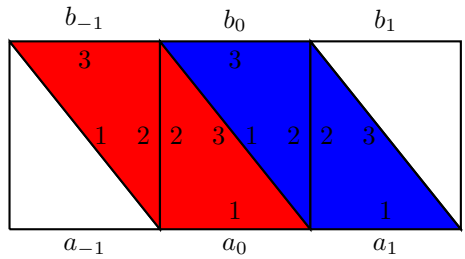


\begin{figure}[htp]
\centering
\begin{tikzpicture}[scale=0.6]
\draw[] (0,0) -- (20,0) -- (20,5) -- (0,5) -- (0,0);
\draw[color=black,style=thick] (0,0) -- (2,0) -- (0,5) -- (0,0);
\draw[color=black,style=thick] (2,0) -- (4,0)  -- (2,5)  -- (0,5) -- (2,0);
\draw[color=black,style=thick, fill=red] (4,0) --   (4,5) node[pos=.7,left] {$2$}   -- (2,5) node[pos=.5,below] {$3$} node[pos=.5,sloped,above] {$c_{-1}$}   -- (4,0)node[pos=.3,right] {$1$};
\draw[color=black,style=thick, fill=red] (4,0) -- (4,5)  node[pos=.3,right] {$2$}     -- (6,0)  node[pos=.7,left] {$3$}  -- (4,0)  node[pos=.5,above] {$1$}   node[pos=.5,sloped,below] {$a_{0}$}  ;
\draw[color=black,style=thick, fill=green] (6,0) -- (8,0)  node[pos=.5,above] {$1$}   -- (6,5)  node[pos=.5,left] {$4$}  -- (4,5)  node[pos=.5,below] {$3$}   node[pos=.5,sloped,above] {$b_{0}$} -- (6,0)  node[pos=.5,right] {$2$} ;
\draw[color=black,style=thick, fill=blue] (8,0) -- (8,5) node[pos=.7,left] {$3$}  -- (6,5) node[pos=.5,below] {$2$} node[pos=.5,sloped,above] {$c_{0}$}  -- (8,0) node[pos=.3,right] {$1$} ;
\draw[color=black,style=thick, fill=blue] (8,0)  -- (8,5)node[pos=.3,right] {$3$}  -- (10,0) node[pos=.7,left] {$2$} -- (8,0) node[pos=.5,above] {$1$} node[pos=.5,sloped,below] {$a_{1}$} ;
\draw[color=black,style=thick] (10,0) -- (12,0)  -- (10,5)  -- (8,5) -- (10,0);
\draw[color=black,style=thick] (12,0) -- (12,5)  -- (10,5) -- (12,0) ;
\draw[color=black,style=thick] (12,0)  -- (16,0) -- (16,5) -- (16,5) --(16,0)  node[pos=.5,left] {$\cdots$}  ;
\draw[color=black,style=thick] (16,0) -- (16,5)  -- (18,0) -- (16,0) ;
\draw[color=black,style=thick, fill=green] (18,0) -- (20,0) node[pos=.5,above] {$2$}   -- (18,5) node[pos=.5,left] {$3$}  -- (16,5) node[pos=.5,below] {$4$}  node[pos=.5,sloped,above] {$b_{2k+3}$}  -- (18,0) node[pos=.5,right] {$1$};
\draw[color=black,style=thick] (20,0) -- (20,5)  -- (18,5) -- (20,0);
\end{tikzpicture}
\caption{Identifications on the boundary of $\Fe$ in the case $\Ga(2k+3,3,3)^k\times(C_3)^k$.}
\label{fig:p2kc}
\end{figure}
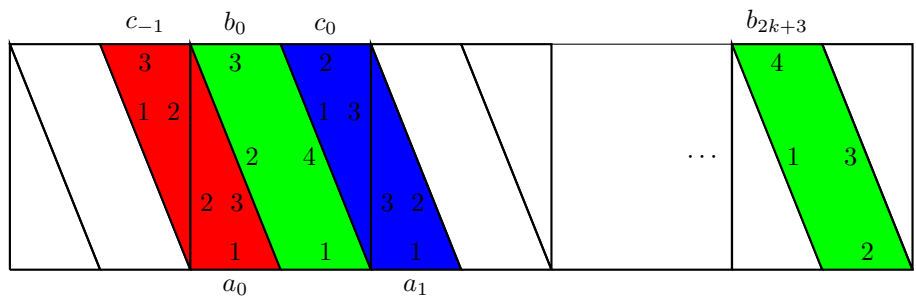

\myskip\noindent
Tables~\ref{tab:fund-k} and~\ref{tab:fund-2k} at the end of the paper show several fundamental domains from each of the infinite series $\Gamma(k+3,3,3)^k\times(C_3)^k$ and $\Gamma(2k+3,3,3)^k\times(C_3)^k$ respectively.
We can see that the combinatorial structure of the fundamental domains and their identification schemes are similar in each of the two series, only the order of the dihedral symmetry increases with~$k$.
For $\Gamma(k+3,3,3)^k\times(C_3)^k$ the fundamental polyhedron has the combinatorial structure of an anti-prism with two triangular faces sharing an edge which are then rotated around the vertical axis.
For $\Gamma(2k+3,3,3)^k\times(C_3)^k$ we see a quadrangular face to which two triangular faces are attached and all three faces are then rotated around the vertical axis.
Note that while the combinatorics of the fundamental domains of all triangle groups in~$\d$ is the same,
the combinatorics of the fundamental domains which we constructed for $\Gamma(k+3,3,3)^k\times(C_3)^k$ and $\Gamma(2k+3,3,3)^k\times(C_3)^k$ in~$\tGrp$ is different.
Hence the combinatorics of our fundamental domains shows the structure of the groups which is not apparent in their fundamental domains in the hyperbolic plane.



\section{Triangle Groups}

\label{triangles}

A triangle group of signature~$(p,q,r)$ is a group of orientation-preserving isometries of the hyperbolic plane,
generated by the rotations through $\frac{2\pi}{p}$, $\frac{2\pi}{q}$, $\frac{2\pi}{r}$
about the vertices of a hyperbolic triangle with angles $\frac{\pi}{p}$, $\frac{\pi}{q}$ and $\frac{\pi}{r}$.
All such groups are conjugate to each other and we will denote such a group $\Ga(p,q,r)$.

\myskip
The following existence result for discrete subgroups of finite level in~$\tGrp$ can be found in~\cite{Pr:diss} (section~2.8, Satz~38).

\begin{prop} 
\label{triangle-lift}
If $\gcd(k,p)=\gcd(k,q)=\gcd(k,r)=1$ and $pqr-pq-qr-rp$ is divisible by~$k$ then there exists a unique subgroup of $\tGrp$ of level~$k$, denoted $\Ga(p,q,r)^k$, such that its image under the projection to~$\text{PSU}(1,1)$ is the triangle group $\Ga(p,q,r)$.
In particular the conditions for the existence of $\Ga(p,3,3)^k$ are that $k$ is not divisible by~$3$ and $p-3$ is divisible by~$k$. 
\end{prop}

\begin{mydef}
For a Fuchsian group $\bGa\in\PSU(1,1)$,
the {\sl edge corona}~$\calE$ with respect to a point $u\in\d$ consists of all those points in $\bGa(u)\bs\{u\}$
whose Dirichlet region shares at least an edge with the Dirichlet region of~$u$,
compare~\cite{GSh}.
\end{mydef}

\noindent
Figures~\ref{fig:gamma533} and~\ref{fig:gamma733} show as an example the edge coronas
for the triangle groups $\Ga(5,3,3)$ and $\Ga(7,3,3)$ respectively.
The Dirichlet region with respect to the origin is red/dark grey,
while the union of Dirichlet regions of the points of the edge corona is green/lighter grey.

\begin{figure}[htp]
\centering
\includegraphics[width=7cm]{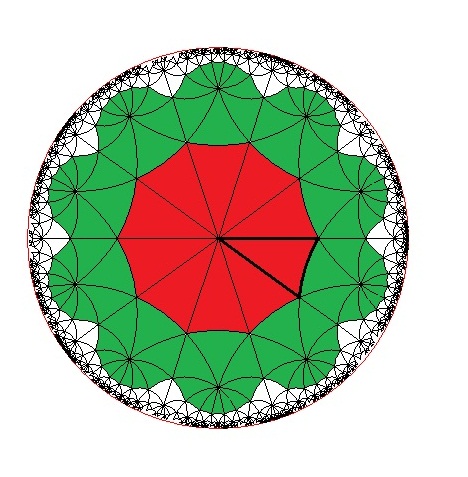}
\caption{The edge corona for $\Ga(5,3,3)$.}
\label{fig:gamma533}
\end{figure}

\begin{figure}[htp]
\includegraphics[width=7cm]{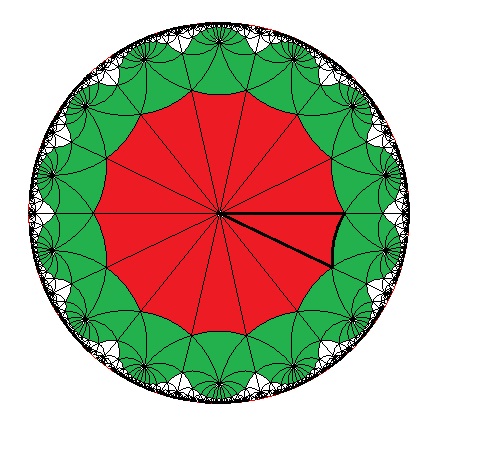}
\caption{The edge corona for $\Ga(7,3,3)$.}
\label{fig:gamma733}
\end{figure}


\myskip
We have the following description of the edge corona,
for details of the proof see Propositions~41 and~42 in~\cite{Pr:diss}:

\begin{prop}
\label{prop-corona}
Let $\Ga(p,q,r)$ be a triangle group
generated by rotations $\rho_u$, $\rho_v$, $\rho_w$ through $\frac{2\pi}{p}$, $\frac{2\pi}{q}$, $\frac{2\pi}{r}$
around the vertices $u$, $v$, $w$ of a hyperbolic triangle with angles $\frac{\pi}{p}$, $\frac{\pi}{q}$, $\frac{\pi}{r}$
respectively.
Let $L$ be the hyperbolic distance between the vertices~$u$ and~$v$.
Recall that $L$ is determined by the angles of the hyperbolic triangle
$$\cosh L=\frac{\cos\frac{\pi}{p}\cos\frac{\pi}{q}+\cos\frac{\pi}{r}}{\sin\frac{\pi}{p}\sin\frac{\pi}{q}}.$$
Then the edge corona consists of the following points of the form $x_{m,l}=(\rho^m_u\rho^l_v)(u)$:
$$\calE=\{x_{0,1},x_{0,q-1},x_{1,1},x_{1,q-1},\cdots,x_{p-1,1},x_{p-1,q-1}\}.$$
All points of the edge corona are on the circle with centre at the origin of (Euclidean) radius
$$d=\frac{\sinh L\sin\frac{\pi}{q}}{\sqrt{\sinh^2 L\sin^2\frac{\pi}{q}+1}}.$$
Moreover, the largest (Euclidean) distance between subsequent points of the edge corona on this circle is
$$s=\frac{\sinh L\sin\frac{2\pi}{q}}{\sinh^2 L\sin^2\frac{\pi}{q}+1}.$$
\end{prop}

\section{The Elements of the Construction}

\label{elements}

\noindent
In this section we will add some notation to the setting that we sketched in the introduction.
Let us describe the embedding of~$\tGrp$ in~$\tLrp$ and the totally geodesic subspaces~$\Eg$ in~$\tLrp$.
We consider the $4$-dimensional pseudo-Euclidean space $E^{2,2}$ of signature $(2,2)$.
We think of $E^{2,2}$ as the real vector space $\c^2\cong\r^4$ with the symmetric bilinear form
$$\<(z_1,w_1),(z_2,w_2)\>=\Re(z_1\bar z_2-w_1\bar w_2).$$
In~$E^{2,2}$ we consider the quadric
$$
  \Grp
  =\left\{a\in E^{2,2}\st\<a,a\>=-1\right\}
  =\left\{(z,w)\in E^{2,2}\st|z|^2-|w|^2=-1\right\}.
$$
We consider the cone over $\Grp$
\begin{align*}
  \Lrp
  &=\{\la\cdot a\st\la>0,~a\in\Grp\}\\
  &=\left\{a\in E^{2,2}\st\<a,a\><0\right\}
  =\left\{(z,w)\in E^{2,2}\st|z|<|w|\right\}.
\end{align*}
The bilinear form on $E^{2,2}$ induces a pseudo-Riemannian metric of signature $(2,1)$ on $\Grp$ and of signature~$(2,2)$ on $\Lrp$.

\myskip
Let $\pi:\tGrp\to\Grp$ be the universal covering.
Henceforth we identify the Lie group $\SU(1,1)$  with $\Grp$ via
$$\WZZW\mapsto(z,w),$$
and $\widetilde{\operatorname{SU}(1,1)}$ with $\tGrp$.
The bi-invariant metrics on $\Grp$ and $\tGrp$ are proportional to the Killing forms.
The covering $\tGrp\to\Grp$ extends to the universal covering $\tLrp\to\Lrp$.
The covering space $\tLrp$ inherits canonically a pseudo-Riemannian metric from~$\Lrp$.
There exist canonical trivializations $\Lrp\cong\Grp\times\r_+$  and $\tLrp\cong\tGrp\times\r_+$.

\myskip
For $g\in\tGrp$, consider the intersection with~$\Lrp$ of the affine tangent space of $\Grp$ at the point $\pi(g)$ 
$$\bE_{\pi(g)}=\{a\in\Lrp\st\<\pi(g),a\>=-1\}$$
and the intersection with~$\Lrp$ of the half-space of $\c^2$ bounded by $\bE_{\pi(g)}$ and not containing the origin
$$\bI_{\pi(g)}=\{a\in\Lrp\st\<\pi(g),a\>\le-1\}.$$
The sets $\bE_{\pi(g)}$ and $\bI_{\pi(g)}$ are simply connected and even contractible,
hence their pre-images under the covering map $\pi$ consist of infinitely many connected components, one of them containing~$g$.
Let $E_g$ and $I_g$ be those connected components of $\pi^{-1}(\bE_{\pi(g)})$ and $\pi^{-1}(\bI_{\pi(g)})$ respectively that contain~$g$.
The three-dimensional submanifold~$\Eg$ divides~$\tLrp$ into two connected components, the closure of one of which is~$\Ig$.
Let $\Hg$ be the closure of the other component of the complement of~$\Eg$ in~$\tLrp$.
The boundaries of~$\Ig$ and of~$\Hg$ are equal to~$\Eg$.

\myskip
The set $\pi((\Ee\cap\dd\Qu)^\circ)$ can be described as
$$\pi((\Ee\cap\dd\Qu)^\circ)=\left\{(z,w)\st\Re(w)=1,~|\Im(w)|<\tan\thk\right\}.$$


\section{Approximating prisms $\Qx$ by their inscribed and subscribed cylinders}

\label{prisms-cylinders}

\noindent
The following estimates are obtained by approximating the prisms~$\Qx$ via their inscribed and subscribed cylinders:

\begin{prop}
\label{in-out-cylinders}
Let $x\in\Gauou$ and $(z,w)\in\Lrp$.
Then
\begin{enumerate}[(i)]
\item
If $(z,w)\in\pi(\Qx)$ then $|w|-|z|<|w|-|x|\cdot|z|\le|w-\bar{x}z|\le\sec\thk\cdot\sqrt{1-|x|^2}$.
\item
If $|w-\bar{x}z|\le\sqrt{1-|x|^2}$ then $\pi^{-1}(z,w)\subset\Qx$.
\end{enumerate}
\end{prop}

\begin{proof}
For the proof of~(i) see Lemma~1(i) in~\cite{Pr:qgor-fund}. The same idea works for~(ii).

Let us first consider the case $x=u=0$.
In this case the inequality in~(ii) reduces to $|w|\le1$.
Let~$g=(z,w)\in\Lrp$ with~$|w|\le1$.
Then we have~$g\in\bar H_a$ for all~$a\in(\bGa_1)_u$ and therefore~$\pi^{-1}(g)\subset H_a$ for all~$a\in(\Ga_1)_u$.
Thus~$\pi^{-1}(g)\subset\Qu$.

In the general case $x\in\Ga_1(u)\bs\{u\}$,
let $g\in\Ga_1$ be an element such that $g(x)=u$ and let $\pi(g)=(a,b)$.
The element $(a,b)\in\Grp$ acts on~$\d$ by
$$(a,b)\cdot x=\frac{\bar b x+a}{\bar a x+b}.$$
The property $(a,b)\cdot x=u=0$ implies $a=-\bar bx$.
From $(a,b)\in\Grp$ we conclude
$$-1=|a|^2-|b|^2=|-\bar b x|^2-|b|^2=-|b|^2\cdot(1-|x|^2)$$
and hence $|b|=(1-|x|^2)^{-1/2}$.
Let us consider $(z,w)$ with $|w-\bar{x}z|\le\sqrt{1-|x|^2}$.
Let $(z',w')=\pi(g)\cdot(z,w)$.
Then
$$
  |w'|
  =|\bar az+b w|
  =|-b\bar xz+bw|
  =\frac{|w-\bar{x}z|}{\sqrt{1-|x|^2}}
  \le1,
$$
hence $\pi^{-1}(z',w')\subset\Qu$
and
$$
  \pi^{-1}(z,w)
  =\pi^{-1}(\pi(g)^{-1}\cdot(z',w'))
  =g^{-1}\cdot\pi^{-1}(z',w')
  \subset g^{-1}\cdot\Qu
  =\Qx.\qedhere
$$ 
\end{proof}

\section{Reduction of the description of $\Fe$}

\label{sec:neclace}

\noindent
The interior of the fundamental domain~$\Fe$ is
$$\Fe^{\circ}=(\Ee\cap\dd\Qu)^\circ-\bigcup_{x\in\Gauou}\Qx.$$
The aim of this section is to show that in the case $\Ga_1=\Ga(p_1,q,r)^k$ and $\Ga_2=(C_{p_2})^k$ with $p=\lcm(p_1,p_2)\ge kq/2$
it is sufficient to consider the prisms~$\Qx$ with $x$ in the edge corona~$\calE$, i.e.\ $\Fe^{\circ}=\FE$, where
$$\FE=(\Ee\cap\dd\Qu)^\circ-\bigcup_{x\in{\calE}}\Qx.$$
We will separate the set $\pi(\FE)$ from the sets $\pi(\Qx)$ for sufficiently large~$|x|$ by enclosing them within cylinders.
We will first describe a cylinder around~$\pi(\FE)$.
To this end we are going to derive an estimate for the distance from the vertical axis to the points in $\pi(\FE)$
by approximating the sets $\pi(Q_x)$ through their inscribed cylinders.

\begin{prop}
\label{prop:5}
Let the functions $\ellE^{\pm}$ be defined as
$$\ellE^{\pm}(t)=\frac{1}{\tanh L}\pm\frac{1}{\sinh L\cdot\sin\frac{\pi}{q}}\cdot\sqrt{\frac{1}{t^2}-\cos^2\frac{\pi}{q}},$$
where $L$ is the distance defined in Proposition~\ref{prop-corona}.
Suppose that
$$\ellE^{-}\left(\sec\thk\right)\le1.$$
Then for any $(z,w)\in\pi((\Ee\cap\dd\Qu)^\circ)$
$$(z,w)\in\pi(\FE)\Longrightarrow|z|<|w|\cdot\ellE^{-}(|w|).$$
\end{prop}

\begin{proof}
Proposition~\ref{prop-corona} says that the points of the edge corona $\calE$ are all on a circle.
Let $x_1,\cdots,x_m$ be the points of $\calE$ in the anti-clockwise direction.
Recall that
$$
  |x_i|=d=\frac{\sinh L\sin\frac{\pi}{q}}{\sqrt{\sinh^2 L\sin^2\frac{\pi}{q}+1}}
  \quad\text{and}\quad
  \max|x_i-x_{i+1}|=s=\frac{\sinh L\sin\frac{2\pi}{q}}{\sinh^2 L\sin^2\frac{\pi}{q}+1}.
$$
Let $z_i=w/\bar x_i$ and $r=\sqrt{1-d^2}/d$.
Proposition~\ref{in-out-cylinders}(ii) implies that $(z,w)\in\pi(Q_x)$ for any $z\in B(z_i,r)$ with $|z|<|w|$.
Let $\hat{d}=|z_i|=|\frac{w}{x_i}|=\frac{|w|}{d}$ and
\begin{eqnarray*}
  2\hat{s}
  =|z_i-z_{i+1}|
  &=&\left|\frac{w}{\bar{x}_i}-\frac{w}{\bar{x}_{i+1}}\right|
  =|w|\cdot\left|\frac{\bar{x}_{i+1}-\bar{x_i}}{\bar{x_i}\bar{x}_{i+1}}\right|
  =|w|\cdot\frac{s}{d^2}.
\end{eqnarray*}
We want to show that $\hat{s}\le r$, i.e.\ that the neighbouring disks $B(z_i,r)$ and $B(z_{i+1},r)$ intersect. 
Straightforward computation shows that
$$
  \hat{s}=|w|\cdot\frac{s}{2d^2}=|w|\cdot\frac{\cos\frac{\pi}{q}}{\sinh L\sin\frac{\pi}{q}}
  \quad\text{and}\quad
  r=\frac{\sqrt{1-d^2}}{d}=\frac{1}{\sinh L\sin\frac{\pi}{q}}.
$$
We know that $(z,w)\in\pi((\Ee\cap\dd\Qu)^\circ)$, hence $\Re(w)=1$ and $|\Im(w)|\le\tan\thk$.
Then,
$$|w|=\sqrt{\Re(w)^2+\Im(w)^2}\le\sqrt{1+\tan^2\thk}=\sec\thk\le\sec\frac{\pi}{q}$$
since $\thk\le\frac{\pi}{q}$.
Therefore
$$\hat{s}=|w|\cdot\frac{\cos\frac{\pi}{q}}{\sinh L\sin\frac{\pi}{q}}\le\frac{1}{\sinh L\sin\frac{\pi}{q}}=r.$$

\begin{figure}[h]
  \begin{center}
  \leavevmode
  \setcoordinatesystem units <0.5cm,0.5cm> 
  \startrotation by 0 -1
  \plot -3 0  1.5 2  3 0  1.5 -2  -3 0 /
  \plot 0 0  1.5 2  1.5 -2  0 0 /
  \plot -3 0  3 0 /
  \put {$\scriptstyle r$} [rb] <0pt,0pt> at 0.5 -1
  \put {$\scriptstyle \hat d$} [r] <0pt,0pt> at -1.5 -1
  \put {$\scriptstyle \hat s$} [t] <3pt,-3pt> at 1.5 -1
  \put {$\scriptstyle z_1$} [r] <-5pt,0pt> at 1.5 -2
  \put {$\scriptstyle z_2$} [l] <5pt,0pt> at 1.5 2
  \circulararc 360 degrees from -1 2 center at 1.5 2
  \circulararc -360 degrees from -1 -2 center at 1.5 -2
  \multiput {\phantom{$\bullet$}} at -3 4.5 4 -4.5 /
  \stoprotation
  \end{center}
\caption{Inscribed cylinders}
\label{figure-neclace}
\end{figure}
Consider the circles of radius~$r$ with centers at~$z_i$ and~$z_{i+1}$.
The distance between their centres is at most $2\hat{s}\le2r$, hence the circles intersect at two points,
see Figure~\ref{figure-neclace}.
The distances of these points to the origin are
\begin{align*}
  &\sqrt{\hat{d}^2-\hat{s}^2}\pm\sqrt{r^2-\hat{s}^2}\\
  &=\sqrt{\frac{|w|^2}{d^2}-\frac{s^2 |w|^2}{4d^4}}\pm\sqrt{\frac{1-d^2}{d^2}-\frac{s^2|w|^2}{4d^4}}\\
  &=|w|\cdot\left(\frac{\sqrt{4d^2-s^2}}{2d^2}\pm\frac{1}{2d^2}\cdot\sqrt{\frac{4d^2(1-d^2)}{t^2}-s^2}\right)\\
  &=|w|\cdot\left(\frac{1}{\tanh L}\pm\frac{1}{\sinh L\cdot\sin\frac{\pi}{q}}\cdot\sqrt{\frac{1}{t^2}-\cos^2\frac{\pi}{q}}\right)\\
  &=|w|\cdot\ellE^{\pm}(|w|).
\end{align*}
This means that any point~$z\in\c$ with $\frac{|z|}{|w|}$ between $\ellE^{-}(|w|)$ and $\ellE^{+}(|w|)$
is contained in the union of the disks $B(z_i,r)$.
Functions $\ellE^-$ and $\ellE^+$ are monotone increasing and decreasing respectively,
therefore any point~$z$ with $\frac{|z|}{|w|}$ between $\ellE^{-}\left(\sec\thk\right)$ and $\ellE^{+}\left(\sec\thk\right)$
is contained in the union of the disks $B(z_i,r)$ and hence in $\bigcup_{x\in\calE}\pi(Q_x)$.
It follows that $|z|<|w|\cdot\ellE^-(|w|)$ for any point $(z,w)$ in~$\pi(\FE)$.
\end{proof}

\myskip
We have now enclosed $\FE$ within a cylinder, it remains to estimate the position of~$\Qx$ for sufficiently large~$|x|$.

\begin{prop}
\label{prop:4.10}
Let the function~$f$ be defined as
$$f(s,t)=\frac{1}{s}-\frac{\sec\thk}{t}\cdot\frac{\sqrt{1-s^2}}{s}.$$
Let $x$ be a point in $\Ga_1(u)\bs\{u\}$ such that $|x|\ge R>0$.
Then for any $(z,w)\in\pi((\Ee\cap\dd\Qu)^\circ)$
$$(z,w)\in\pi(Q_x)\Longrightarrow|z|\ge|w|\cdot f(R,|w|).$$
\end{prop}

\begin{proof}
Let $(z,w)$ be a point in $\pi(Q_x)$.
Proposition~\ref{in-out-cylinders}(i) implies that
$$|w|-|x|\cdot|z|\le|w-x\bar{z}|\le\sec\thk\cdot\sqrt{1-|x|^2},$$
hence
$$|z|\ge\frac{|w|}{|x|}-\frac{\sec\thk\sqrt{1-|x|^2}}{|x|}=|w|\cdot f(|x|,|w|).$$
Note that for $s\in(0,1)$ and $t\in\left[1,\sec\thk\right]$ we have
$$\frac{\partial f}{\partial s}=\frac{\frac{\sec\thk}{t}-\sqrt{1-s^2}}{s^2\sqrt{1-s^2}}\ge0,$$
hence $f(\cdot,t)$ is monotone increasing for $t\in\left[1,\sec\thk\right]$.

\myskip
For $(z,w)\in\pi((\Ee\cap\dd\Qu)^\circ)\cap\pi(Q_x)$ we have $|w|\in\left[1,\sec\thk\right]$ and therefore $f(|x|,|w|)>f(R,|w|)$ and
$$|z|\ge|w|\cdot f(|x|,|w|)>|w|\cdot f(R,|w|).\qedhere$$
\end{proof}

In Proposition~\ref{prop:5} we found an upper bound on the distance from the vertical axis to the points of the set $\pi(\FE)$.
On the other hand, in proposition~\ref{prop:4.10} we provide a lower bound for the distance from the vertical axis
to the points of the set $\pi(Q_x)$.
Combining these two estimates we can show under certain conditions that the sets $\pi(Q_x)$ with $x\not\in\calE\cup\{u\}$
share no points with $\pi(\FE)$ and therefore with $\pi(F_{{e}})\subset\pi(\FE)$.
That would mean $\pi(F_e)=\pi(\FE)$ and hence $F_e=\FE$.\\

\begin{thm}

\label{prop:7.1}

Let $R\ge\tanh L$ be such that
$$\{x\in\Ga_1(u)\st|x|<R\}=\calE\cup\{u\}$$
and
$$\ellE^{-}\left(\sec\thk\right)\le\frac{1-\sqrt{1-R^2}}{R}.$$
Then
$$\Fe^\circ=\FE.$$
\end{thm}

\begin{proof}
Let $x$ be a point in $\Ga_1(u)\bs\{u\}$ such that $|x|\ge R$.
Let $(z,w)\in\pi((\Ee\cap\dd\Qu)^\circ)$.
We have $\ellE^-\left(\sec\thk\right)\le\frac{1-\sqrt{1-R^2}}{R}\le 1$.
According to Proposition~\ref{prop:5},
$$(z,w)\in\pi(\FE)\Longrightarrow|z|<|w|\cdot\ellE^-(|w|).$$
On the other hand Proposition~\ref{prop:4.10} implies that
$$(z,w)\in\pi(Q_x)\Longrightarrow|z|\ge|w|\cdot f(R,|w|).$$
So to prove that $\pi(Q_x)\cap\pi(\FE)=\emptyset$ for all $x\in\Ga_1(u)$ with $|x|\ge R$,
it is sufficient to show that
$$f(R,t)-\ellE^{-}(t)\ge0\quad\text{for}~t\in\left[1,\sec\thk\right].$$
Expressing the function $f(R,t)-\ellE^{-}(t)$ as
$$
  \frac{1}{t}\cdot\left(\frac{1}{\sinh L\cdot\sin\frac{\pi}{q}}\cdot\sqrt{1-t^2\cdot\cos^2\frac{\pi}{q}}-\sec\thk\cdot\frac{\sqrt{1-R^2}}{R}\right)
  +\left(\frac{1}{R}-\frac{1}{\tanh L}\right)
$$
we see that it is monotone decreasing on $\left[1,\sec\frac{\pi}{q}\right]$ and hence on $\left[1,\sec\thk\right]$ as $\thk\le\frac{\pi}{q}$.
Therefore for all $t\in[1,\sec\thk]$ we have
\begin{align*}
  f(R,t)-\ellE^{-}(t)
  &\ge f\left(R,\sec\thk\right)-\ellE^{-}\left(\sec\thk\right)\\
  &=\frac{1-\sqrt{1-R^2}}{R}-\ellE^-\left(\sec\thk\right)
  \ge0.
\end{align*}
Thus,we proved that
$$\pi(Q_x)\cap\pi(\FE)=\emptyset$$
for all $x\in\Ga_1(u)$ such that $|x|\ge R$,
i.e.\ for all $x\in\Ga_1(u)\bs\left(\calE(u)\cup\{u\}\right)$.
Therefore,
$$\Fe^\circ=\FE.\qedhere$$
\end{proof}

\section{Proof of Theorem~\ref{thm-corona}}

In this section we will apply Theorem~\ref{prop:7.1} to the special case $\Ga_1=\Ga(p_1,3,3)^k$ with $p_1\ge k+3$ and $\Ga_2=(C_3)^k$ to prove Theorem~\ref{thm-corona}.
We have $p=\lcm(p_1,3)=3p_1$.
Let $\al=\frac{\pi}{2p_1}\in\left(0,\frac{\pi}{10}\right)$.
In the case $q=r=3$ the formulas in Proposition~\ref{prop-corona} become
$$
  \cosh L
  =\frac{1}{\sqrt{3}}\cdot\frac{\cos\al}{\sin\al},~
  \sinh L
  =\frac{1}{\sqrt{3}}\cdot\frac{1}{\sin\al}\cdot\sqrt{\frac{\cos(3\al)}{\cos\al}},~
  \tanh L
  =\frac{1}{\cos\al}\cdot\sqrt{\frac{\cos(3\al)}{\cos\al}}.
$$
To use Theorem~\ref{prop:7.1} in this case,
we need to find $R\ge\tanh L$ such that
$$\{x\in\Ga_1(u)\st|x|<R\}=\calE\cup\{u\}\quad\text{and}\quad\ellE^-\left(\sec\thk\right)\le\frac{1-\sqrt{1-R^2}}{R}.$$
A careful study of the structure of the orbit $\Ga_1(u)$ (for details see Proposition~72 in~\cite{Pr:diss})
shows that $|x|\ge R$ for every point~$x\in\Ga_1(u)\backslash(\calE\cup\{u\})$,
where
$$R=\frac{\cos\al}{\cos(2\al)}\cdot\sqrt{\frac{\cos(3\al)}{\cos\al}}.$$
It is easy to check that $R>\tanh L$.
We then compute
\begin{align*}
  &\ellE^-\left(\sec\thk\right)=\left(\cos\al-\sin\al\cdot\sqrt{4\cos^2\thk-1}\right)\cdot\sqrt{\frac{\cos\al}{\cos(3\al)}},\\
  &\frac{1-\sqrt{1-R^2}}{R}=\frac{\cos(2\al)-\sin\al}{\cos\al}\cdot\sqrt{\frac{\cos\al}{\cos(3\al)}},
\end{align*}
hence the inequality $\ellE^-\left(\sec\thk\right)\le\frac{1-\sqrt{1-R^2}}{R}$ is equivalent to
$$
  \sqrt{4\cos^2\thk-1}
  \ge\frac{\cos^2\al-\cos(2\al)+\sin\al}{\sin\al\cos\al}
  =\cot\left(\frac{\pi}{4}-\frac{\al}{2}\right)
  =\sqrt{\csc^2\left(\frac{\pi}{4}-\frac{\al}{2}\right)-1},
$$
which is equivalent to
$$2\cos\thk\ge\csc\left(\frac{\pi}{4}-\frac{\al}{2}\right).$$
We have
$$
 0<\thk
 =\frac{k\pi}{6p_1}
 \le\frac{(p_1-3)\pi}{6p_1}
 =\frac{\pi}{6}-\al
  \le\frac{\pi}{4}-\frac{3\al}2
  <\frac{\pi}{4}
$$
since $k\le p_1-3$ and $\al<\pi/6$.
Thus
$$2\cos\thk\ge2\cos\left(\frac{\pi}{4}-\frac{3\al}{2}\right).$$
It remains to show that 
$$2\cos\Big(\frac{\pi}{4}-\frac{3\al}{2}\Big)\ge\csc\Big(\frac{\pi}{4}-\frac{\al}{2}\Big).$$
Note that
$$
  2\cos\Big(\frac{\pi}{4}-\frac{3\al}{2}\Big)\sin\Big(\frac{\pi}{4}-\frac{\al}{2}\Big)
  =\cos(2\al)+\sin\al
  \ge1
$$
since $\sin\al\in(0,\frac12)$.
Thus we have checked that the conditions of Theorem~\ref{prop:7.1} are satisfied in this case, hence $\Fe^\circ=\FE$.


\begin{table}
\centering
\begin{center}
    \begin{tabular}{|c|c|}
    \hline
     $k=4$ & \includegraphics[height=4cm]{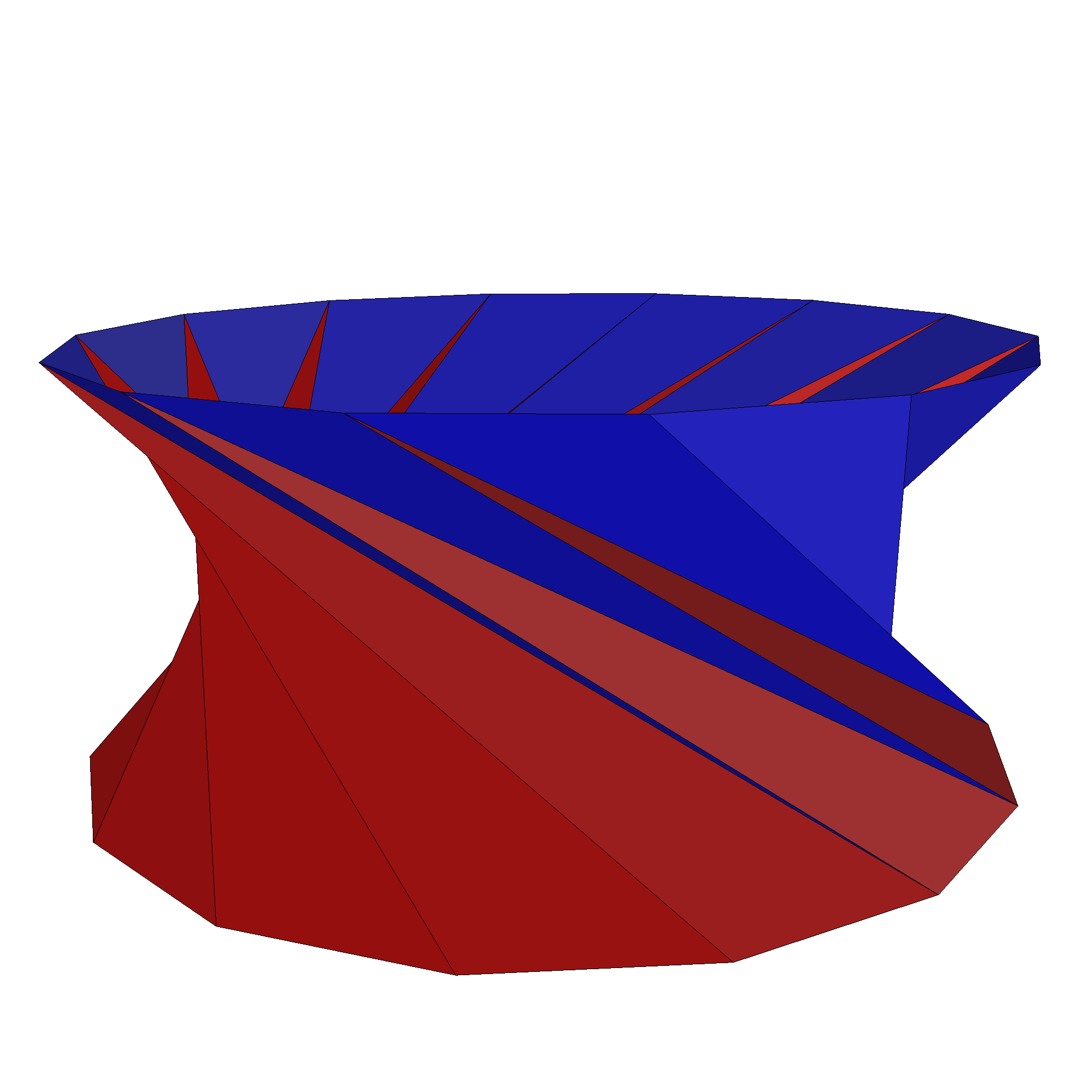} \\  \hline
     $k=5$ & \includegraphics[height=4cm]{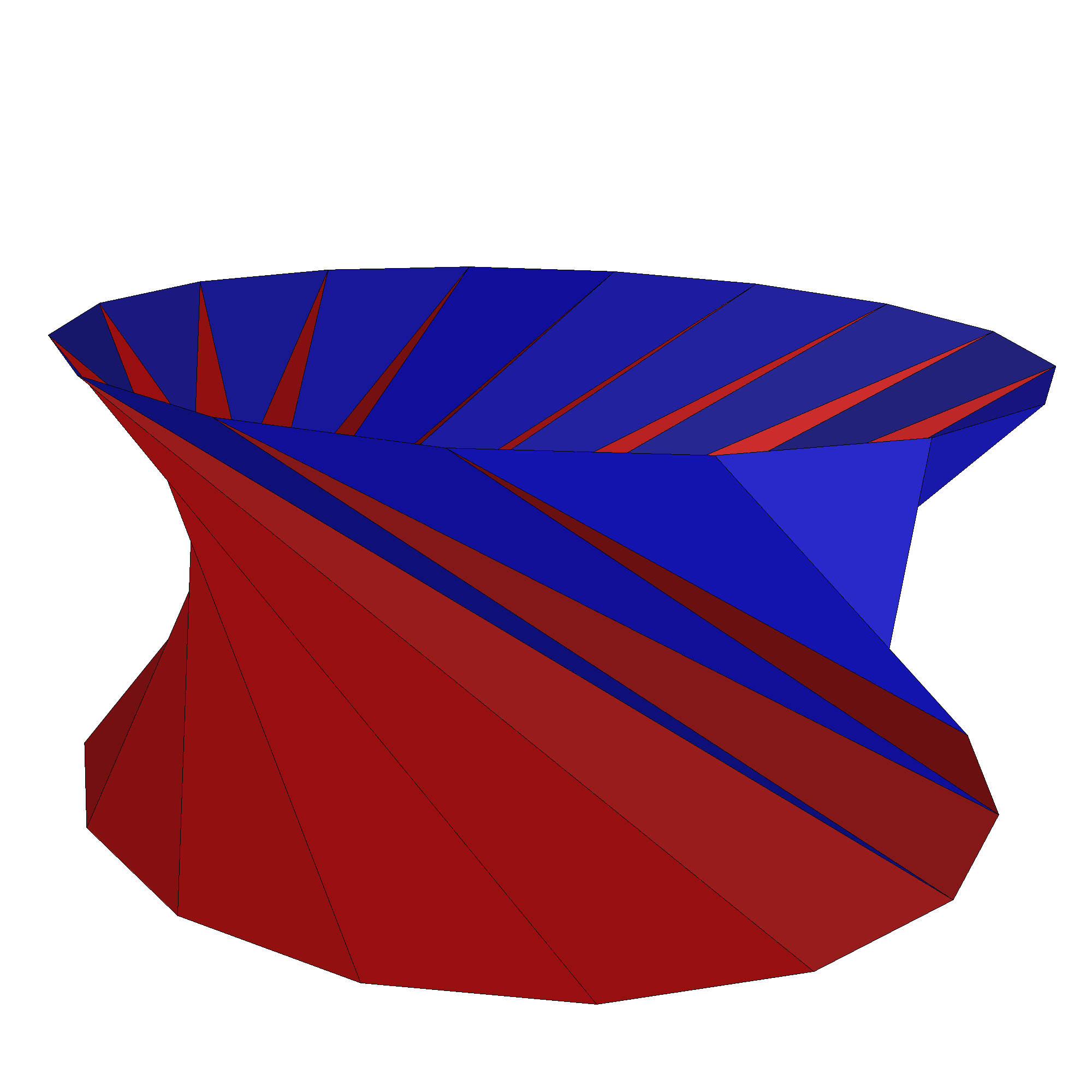} \\  \hline
     $k=7$ & \includegraphics[height=4cm]{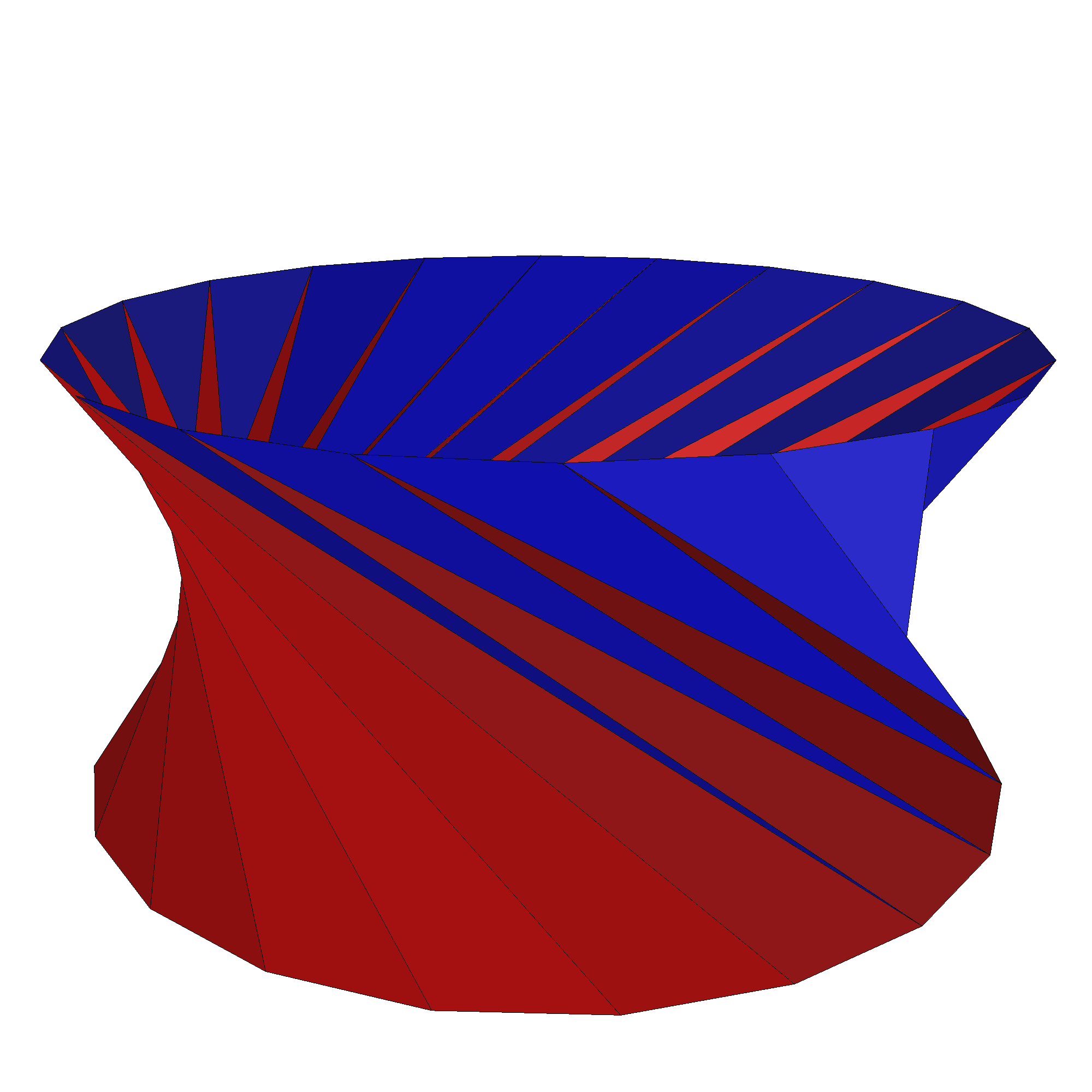} \\ \hline
     $k=8$ & \includegraphics[height=4cm]{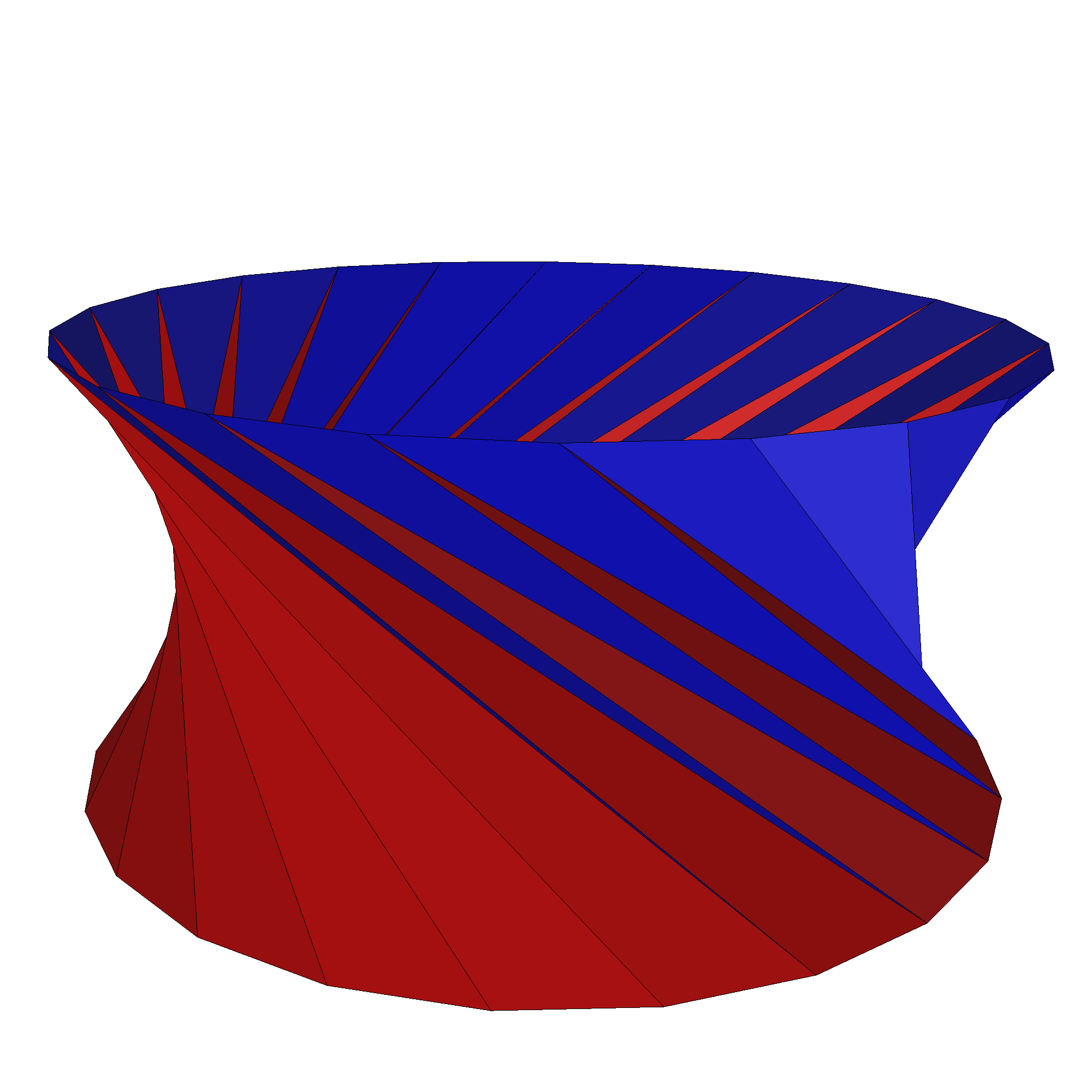} \\ \hline
    \end{tabular}
\end{center}
\caption{Fundamental domains for $\Gamma(k+3,3,3)^k\times(C_3)^k$.}
\label{tab:fund-k}
\end{table}

\begin{table}
\centering
\begin{center}
    \begin{tabular}{|c|c|}
    \hline
     $k=4$ & \includegraphics[width=5cm]{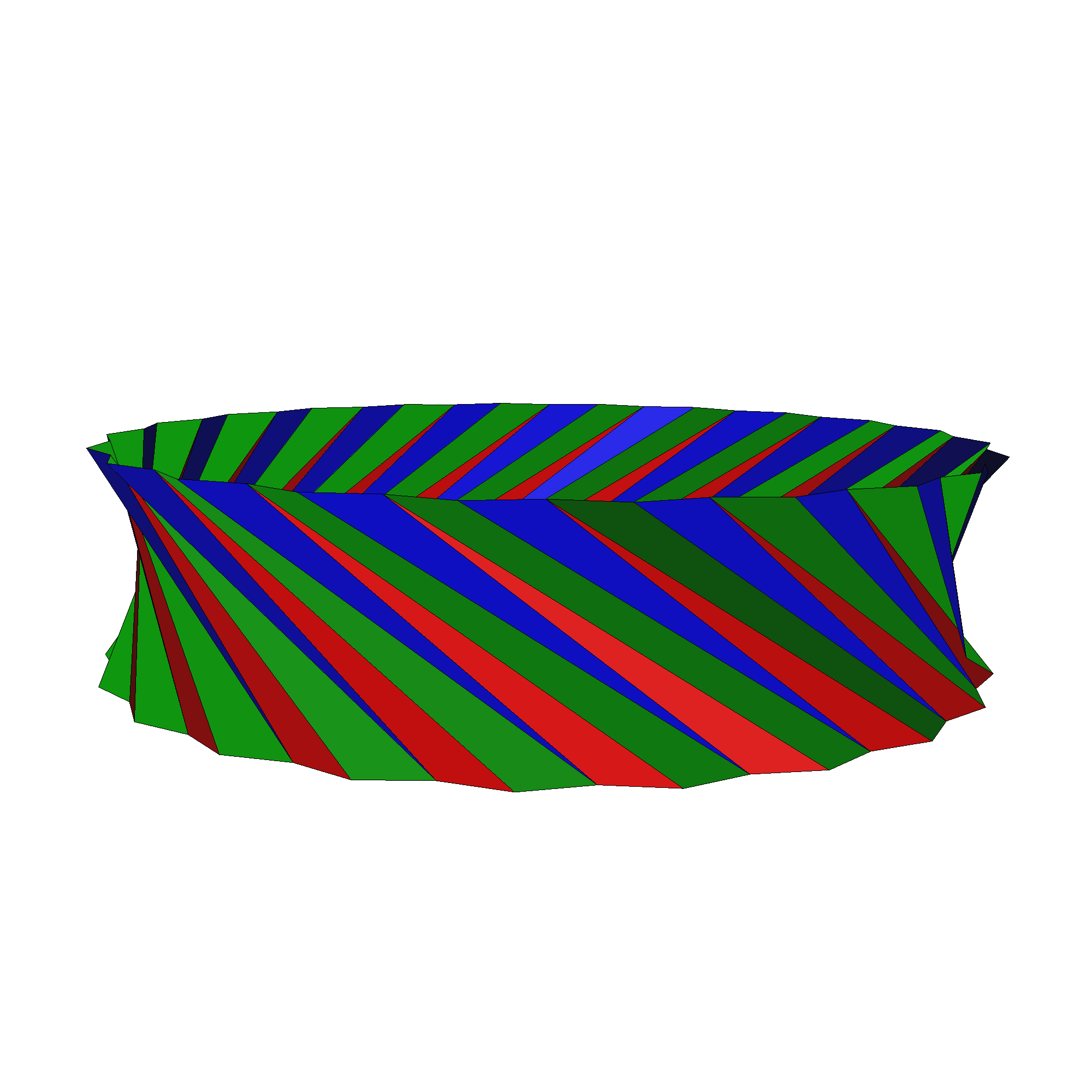} \\ \hline
     $k=5$ & \includegraphics[width=5cm]{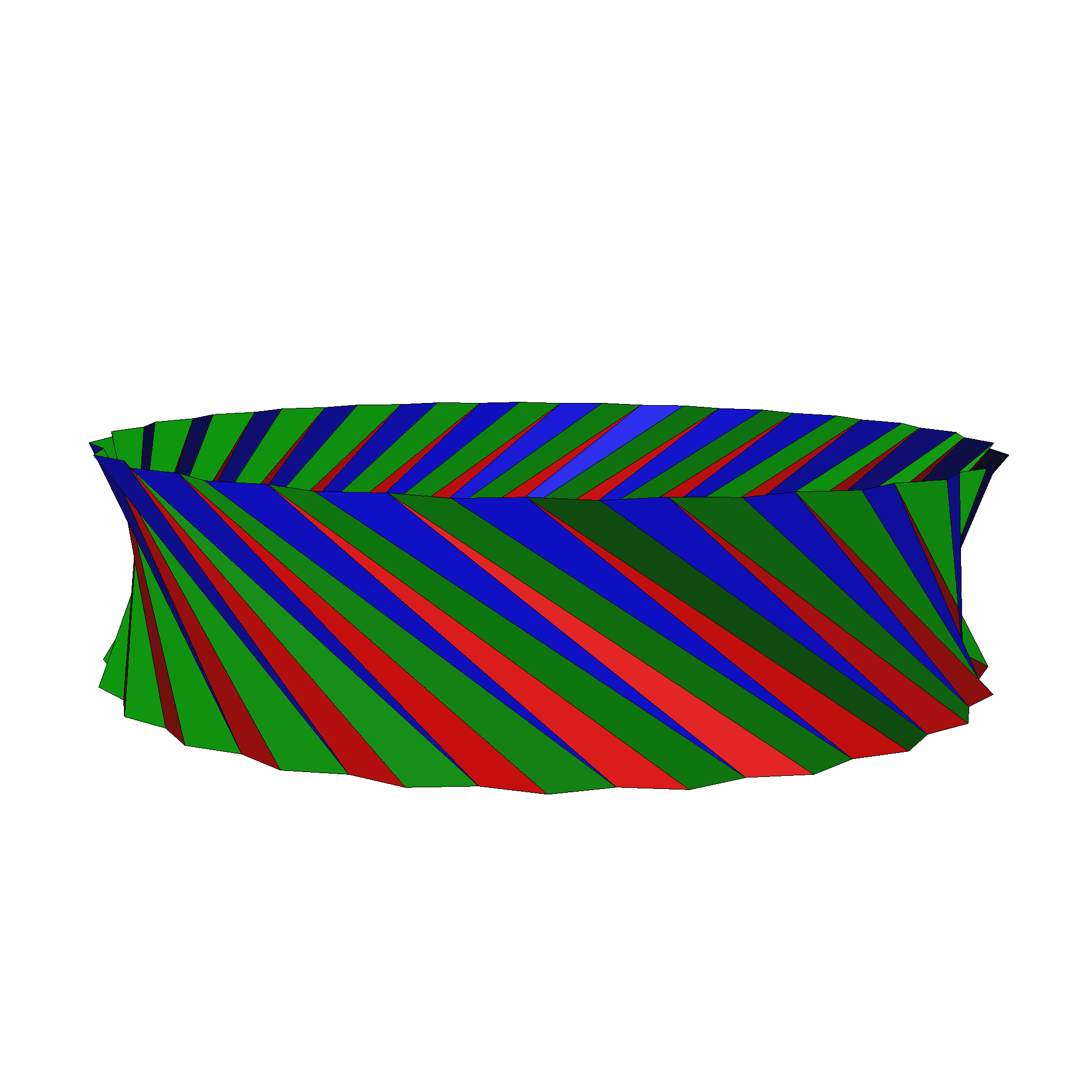} \\ \hline
     $k=7$ & \includegraphics[width=5cm]{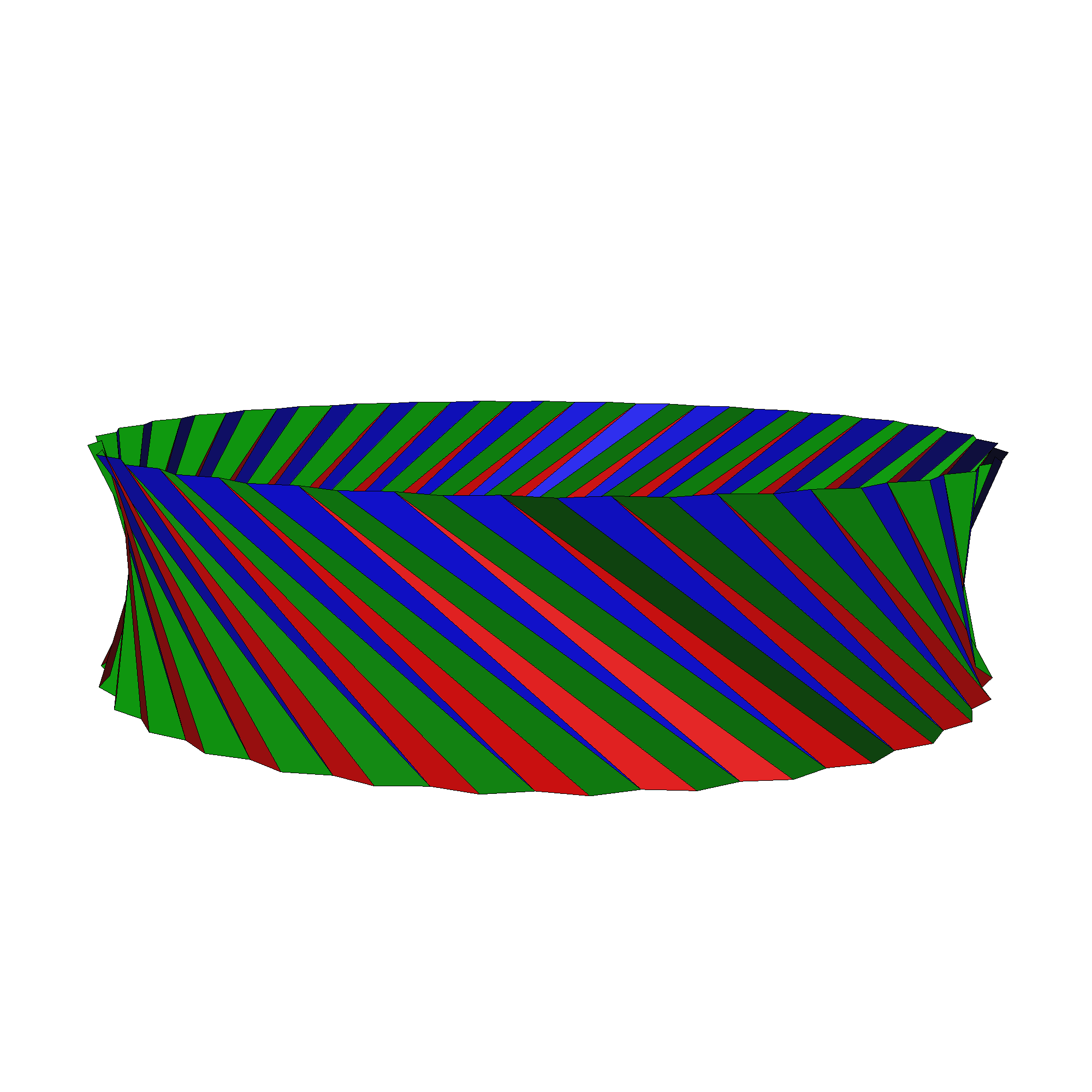} \\ \hline
     $k=8$ & \includegraphics[width=5cm]{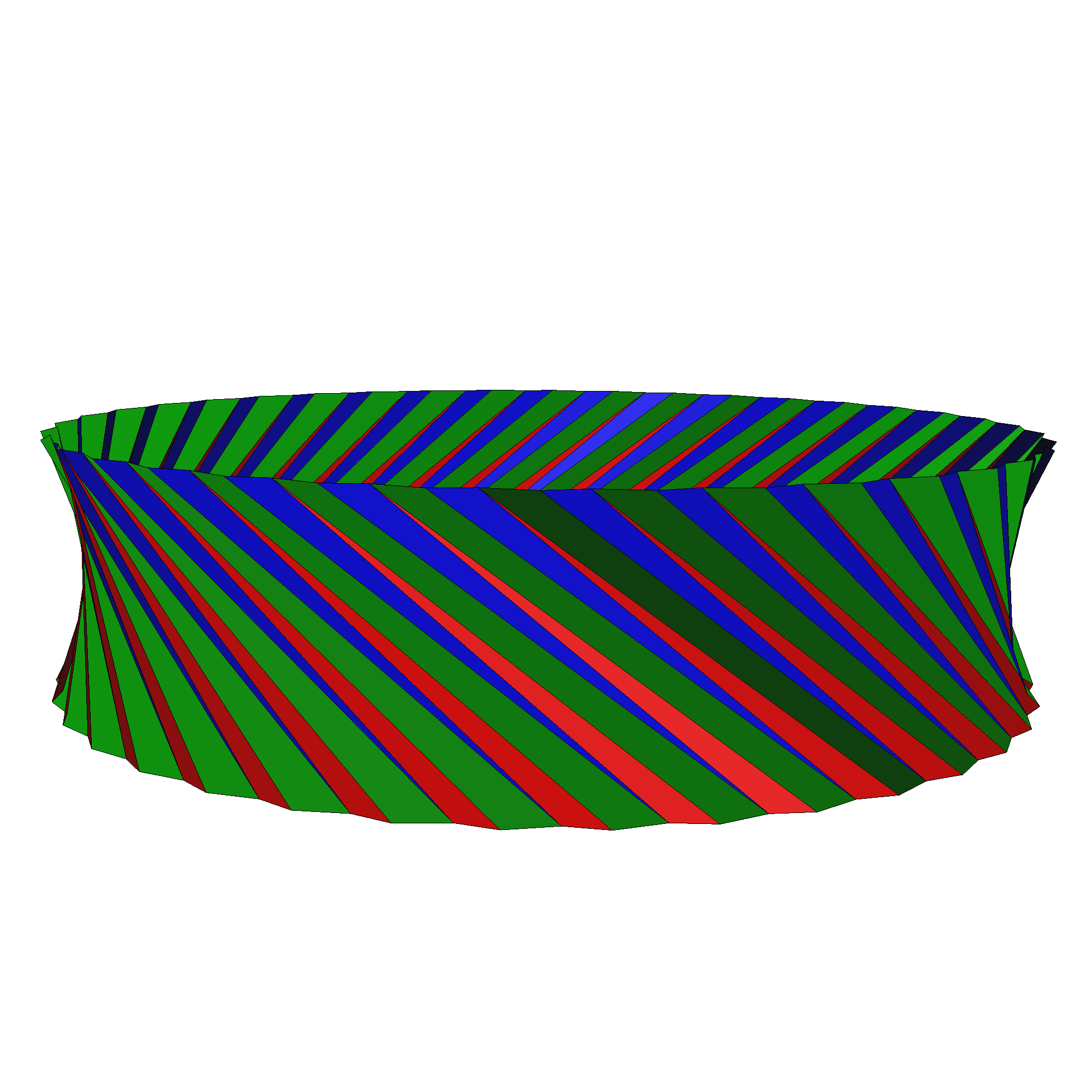} \\ \hline
    \end{tabular}
\end{center}
\caption{Fundamental domains for $\Gamma(2k+3,3,3)^k\times(C_3)^k$.}
\label{tab:fund-2k}
\end{table}

\bigskip\noindent
{\bf Acknowledgements:}
We would like to thank the referees for their helpful comments.


\providecommand{\bysame}{\leavevmode\hbox to3em{\hrulefill}\thinspace}
\providecommand{\MR}{\relax\ifhmode\unskip\space\fi MR }
\providecommand{\MRhref}[2]{%
  \href{http://www.ams.org/mathscinet-getitem?mr=#1}{#2}
}
\providecommand{\href}[2]{#2}

\end{document}